\documentclass[12pt]{article}
\usepackage{hyperref}
\usepackage[latin1]{inputenc}
\usepackage{amsfonts}
\usepackage{amsmath,amssymb,amsthm}
\usepackage{color}
\usepackage[english]{babel}%
\usepackage[shortlabels]{enumitem}
\newtheorem{theorem}{Theorem}[section]

\newtheorem{lema}[theorem]{Lemma}
\newtheorem{de}[theorem]{Definition}
\newtheorem{remark}[theorem]{Remark}

\numberwithin{equation}{section}

\def \eps {\varepsilon}

\title{\vspace{-0.in}\parbox{\linewidth }{\footnotesize\noindent
} \\  \bf Fractional diffusion-wave equations with critical nonlinearities in  Lebesgue spaces}
\author{
Masterson Costa$^1$, Claudio Cuevas$^1$ \& Bruno de Andrade$^2$\footnote{Corresponding author. E-mail addresses: masterson.costa@dmat.ufpe.br (M. C.), claudio.henriquez@ufpe.br (C. C.), bruno@mat.ufs.br (B. A.).}\\
\small{$^1$ Department of Mathematics, Federal University of Pernambuco, Recife-PE, CEP 50.540-740, Brazil.}\\
\small{$^2$ Department of Mathematics, Federal University of Sergipe, Sergipe-SE, CEP 49.000-100, Brazil.}
}
\date{ }
\hypersetup{
colorlinks = true,
linkcolor = black,
anchorcolor = blue,
citecolor = blue,
filecolor = blue,
urlcolor = blue
}

\begin{document}

\maketitle

\author{ \ }

\begin{abstract}
	\noindent This paper focuses on the study of semilinear fractional diffusion-wave equations in the context of critical nonlinearities. Firstly, we address the issue of local well-posedness for the problem, examine spatial regularity, and the continuous dependence of the solutions on initial data. Secondly, we establish the existence of global mild solutions and investigate their asymptotic behavior.

\end{abstract}

\pagestyle{myheadings} \markboth{ } {\hfil$\hspace{1.5cm}$ {} \hfil}


\noindent{\textbf{MSC Classifications 2020:} 	35R11; 35B33; 35B65.}

\section{Introduction}\label{section1}


 It is well established that partial differential equations involving fractional derivatives
arise in a broad range of mathematical and applied contexts. Applications of fractional evolution equations are mentioned in the modeling of anomalous diffusion in  a wide variety of natural sciences, including physics, chemistry, biology, engineering, geology, and their interfacial disciplines, see \cite{Bouchaud,Fre, FMainardi, MK, MIMP, Uch} and the references therein.  Among the various models within this theory, the fractional diffusion-wave equations have attracted significant attention from researchers worldwide. This interest is largely due to their characteristic of interpolation between diffusion and wave equations. In recent years, substantial progress has been made in the analysis of these equations. We recommend the references \cite{B1, BKP, Dabb, AF,  deAS, Dj, HY, Kian, KKL, P} for further reading. 

Motivated by this increasing interest and the above works, in this paper, we study the fractional diffusion-wave equation
\begin{equation}\label{diffwave}
	\left\{ \begin{array}{ll}
		\partial_{t}^{\alpha} u = \Delta u +f(u),~\mbox{in}~[0,\infty)\times \Omega,\\
		u=0,~\mbox{on}~[0,\infty)\times \partial\Omega,~~\\
		u(0,x) = u_0(x),~ u'(0,x) = u_1(x),~\mbox{in}~\Omega,&
	\end{array} \right.
\end{equation}
where $\Omega\subset \mathbb{R}^{N}$ is a smooth bounded domain, $\alpha \in (1, 2)$ and $\partial_{t}^{\alpha}$ is Caputo's fractional derivative. Furthermore,  $f:\mathbb{R}\to\mathbb{R}$ is a continuous function that verifies $f(0)=0$ and 
\begin{eqnarray}\label{fc}
	|f(r)-f(s)|\leq c(|r|^{\rho-1}+|s|^{\rho-1})|r-s|,\quad \forall r,s\in\mathbb{R},
\end{eqnarray}
for some $c>0$ and $\rho>1$.  Roughly speaking, we are focused on determining the critical exponent for the equation \eqref{diffwave} when the initial data belongs to the space \(L^q(\Omega)\), where \(1 < q < \infty\), or to spaces with lower regularity.  The search for critical exponents is a significant theme in the study of various types of evolution equations. For example, in their pioneering paper \cite{fujita1963navier}, Fujita \& Kato search for the largest fractional power space in which a local-existence and uniqueness theorem holds for the  Navier-Stokes equation in a bounded smooth domain of $\mathbb{R}^3$. Fujita \cite{Fuj} considers the heat equation in $\mathbb{R}^N$ with a nonlinear term satisfying \eqref{fc}. Taking smooth non-negative initial conditions, he proves that $1+\frac{2}{N}$ is critical in the sense that global existence holds for small data if  $\rho>1+\frac{2}{N}$ and finite time blow-up behavior of solutions holds if $1< \rho <1+\frac{2}{N}$. Also, John \cite{John} considers the semilinear wave equation in $\mathbb{R}^3$ with nonlinearity satisfying \eqref{fc}, showing that $1+2^{1/2}$ plays a critical role in the analysis of global existence for compactly supported initial data. In \cite{deAV}, de Andrade \& Viana study the strongly damped plate equation 
	\begin{eqnarray*}
		u_{tt} = -\Delta^2 u +\mu\Delta u_t +|u|^{\rho-1}u,\ t>0,\ x\in\Omega,
	\end{eqnarray*}
where $\mu>0$, $\Delta^2$ is the biharmonic with hinged boundary conditions and $\Delta$ is the Dirichlet-Laplacian  in $L^2(\Omega)$, and prove that $1+\frac{4}{N}$ is the critical exponent if initial conditions are taken in $\big(H^{2}(\Omega)\cap H^{1}_0(\Omega)\big)\times L^2(\Omega)$. 

Concerning the fractional diffusion-wave equation, de Andrade \& Santos \cite{deAS} provided a thorough analysis of the evolution operators associated with \eqref{diffwave} in the fractional power scale related to the Laplace operator.  Considering initial data in $L^q(\Omega)$, they demonstrate that mild solutions of the nonlinear fractional diffusion-wave equation exhibit an immediate regularization effect, in the sense that these solutions possess greater regularity than the initial data. The results in \cite{deAS} are valid for any $1<\rho<1+\frac{2q}{N}$; however, when
\begin{equation}\label{critico}
	\rho=1+\frac{2q}{N},
\end{equation}
the technical tools used in their analysis are no longer sufficient to establish the well-posedness for \eqref{diffwave}, leaving this case as an open problem. The value \eqref{critico} plays a critical role to the study of existence and uniqueness of mild solutions for \eqref{diffwave}. Indeed, a combination of \cite[Theorem 4.1]{deAS} and our Theorem \ref{local} below ensures that for any $f$ satisfying \eqref{fc} and initial data $u_0,u_1\in L^{q}(\Omega)$, $1<q<\infty$ and $q\ge\frac{N(\rho-1)}{2}$, the problem \eqref{diffwave} is well-posed. However, if $\rho>1+\frac{2q}{N}$, or equivalently $q<\frac{N(\rho-1)}{2}$, and initial data $u_0,u_1\in L^{q}(\Omega)$, the problem \eqref{diffwave} has no mild solution for some Lipschitz functions $f$, see Remark \ref{rmkcritico}.

At this point, it is important to observe that \eqref{critico} is related with the critical case for the nonlinear  heat equation
\begin{equation}\label{heat}
	\left\{ \begin{array}{ll}
		u_{t}  = \Delta u +f(u),~\mbox{in}~[0,\infty)\times \Omega,\\
		u=0,~\mbox{on}~[0,\infty)\times \partial\Omega,~~\\
		u(0,x) = u_0(x),~\mbox{in}~\Omega,&
	\end{array} \right.
\end{equation}
where   $f$ satisfies \eqref{fc} and $u_0\in L^{q}(\Omega)$, $1<q<\infty$.  Without being exhaustive, we mention Weissler \cite{weissler79, weissler80} who firstly addressed the subject of existence and nonexistence of solutions to \eqref{heat}. Brezis \& Cazenave \cite{caze} have improved the analysis proving the uniqueness of solutions in a larger class of functions. Furthermore, they highlighted that the value $q=\frac{N(\rho-1)}{2}$ plays a critical role to \eqref{heat}. Arrieta \& Carvalho \cite{arrieta2000}  treated the problem by focusing on the non-linear term. Indeed, they introduced the notion of $\epsilon$-regular map and generalized the results of Brezis \& Cazenave to an abstract framework. In \cite{ACCM}, de Andrade et al. extended some of these results to the subdiffusive heat equation
\begin{equation}\label{sub}
	\left\{ \begin{array}{ll}
		\partial_{t}^{\gamma} u = \Delta u +f(u),~\mbox{in}~[0,\infty)\times \Omega,\\
		u=0,~\mbox{on}~[0,\infty)\times \partial\Omega,~~\\
		u(0,x) = u_0(x),~\mbox{in}~\Omega,&
	\end{array} \right.
\end{equation}
where $\partial_{t}^{\gamma}$ is Caputo's fractional derivative,  $\gamma \in (0,1)$, $f$ satisfies \eqref{fc} and $u_0\in L^{q}(\Omega)$, $1<q<\infty$. Particularly, the value $q=\frac{N(\rho-1)}{2}$, or equivalently $\rho=1+\frac{2q}{N}$, is, as to the case of the heat equation, the critical exponent to the issue of existence and uniqueness of mild solutions for  the problem. These facts and our Theorem \ref{local} below ensure that the critical exponent is independent of the time-fractional derivation order, that is,  superdiffusive \eqref{diffwave}, diffusive \eqref{heat} and subdiffusive \eqref{sub} problems have the same critical Lebesgue space, when the initial data are taken in $L^q(\Omega)$, $1<q<\infty$.  To the best of our knowledge, this fact is presented for the first time in this work.


To be more precise in our discussion, let $\{X^{\gamma}_{q}\}_{\gamma\in\mathbb{R}}$ be the scale of fractional powers spaces associated with the realization of the Laplace operator in $L^q(\Omega)$, with $1 < q <\infty$ and $q = \frac{N(\rho - 1)}{2}$. Firstly, we prove that a function $f$ satisfying \eqref{fc} induces an $\epsilon$-regular map relative to the pair $(X^{1}_{q}, X^{0}_{q})$, see Lemma \ref{reg} below. This property provide the essential framework for establishing the existence of mild solutions in the $L^{q}(\Omega)$ context,  see Section \ref{setting} for details of the abstract setting. 

Our main local result is the following.

\begin{theorem}\label{local}
	Let $1<\alpha<\frac{2\phi_q}{\pi}$, $1<\rho, q<\infty$, such that $q=\frac{N(\rho-1)}{2}$, and  $0<\varepsilon<\frac{N}{N+2q}$,  with $\alpha\rho\varepsilon<1$. Then, for all $v_0 \in L^q(\Omega)$, there exist $r=r(v_0)>0$ and $\tau_0=\tau_0(v_0)>0$ such that the problem \eqref{diffwavep} has an $\varepsilon$-regular mild solution $u(\cdot\ ;u_0,u_1)$ defined on $[0,\tau_0]$, for all $u_0,u_1 \in B_{r} (v_0) \subset L^{q}(\Omega)$. Moreover, the following statements hold:
	\begin{itemize}
		\item[{(A)}] For all $\theta \in [0, \rho\varepsilon)$, we have
		$$u(\cdot\ ;u_0,u_1) \in C\left((0,\tau_0];X^{1+\theta}_{q}\right)$$
		and, if $J \subset B_{r} (v_0) $ is compact, then
		$$\lim_{t\to0^+}t^{\alpha\varepsilon}\sup_{u_0, u_1 \in J}\|u(t;u_0, u_1)\|_{X^{1+\varepsilon}_{q}}=0.$$
		\item[{(B)}] For each $\theta \in [0, \rho\varepsilon)$, there exists a constant $\bar{c}=\bar{c}(\alpha, \rho, \varepsilon, \theta,v_0)$ such that
		$$	t^{\alpha\theta} \|u(t;u_0,u_1)-u(t;w_0,w_1)\|_{X^{1+\theta}_{q}} \le \Bar{c}\left(\|u_0-w_0\|_{L^{q}(\Omega)}+\|u_1-w_1\|_{L^{q}(\Omega)}\right), $$
		for all $t \in(0,\tau_0]$, and $u_0,u_1,w_0,w_1 \in B_{r}(v_0) \subset L^{q}(\Omega)$.
		\item[{(C)}] If $v$ is an $\varepsilon$-regular mild solution on some interval $[0,\tau_1]$ for the problem \eqref{diffwavep} satisfying
		$$\lim_{t\to0^+} t^{\alpha\varepsilon}\|v(t)\|_{X^{1+\varepsilon}_{q}}=0,$$
		then $v(t)=u(t;u_0,u_1)$ for all $t \in [0,\min\{ \tau_1,\tau_0\}]$
		\item[{(D)}] The $\varepsilon$-regular mild solution $u(\cdot\ ;u_0, u_1)$ can be continued on an interval $[0,\tau_{\mathrm{max}})$, where $\tau_{\mathrm{max}} \in (\tau_0,\infty]$. If $\tau_{\mathrm{max}}<\infty$, then 
		$$\limsup_{t \to \tau_{\textrm{max}}\,^-} \|u(t;u_0, u_1)\|_{_{X^{1+\varepsilon}_{q}}}=\infty.$$
	\end{itemize}
\end{theorem}

We apply the previous theorem to an illustrative example. Consider the Cauchy problem
\begin{equation}\label{diffwaveexample}
	\left\{ \begin{array}{ll}
		\partial_{t}^{\alpha} u = \Delta u +u^{3},~\mbox{in}~[0,\infty)\times \Omega,\\
		u=0,~\mbox{on}~[0,\infty)\times \partial\Omega,~~\\
		u(0,x) = u_0(x),~ u'(0,x) = u_1(x),~\mbox{in}~\Omega,&
	\end{array} \right.
\end{equation}
where $\Omega$ is a bounded open subset of $\mathbb{R}^{3}$ with sufficiently smooth boundary $\partial \Omega$, and $\partial_{t}^{\alpha}$ is Caputo's fractional derivative, for $1<\alpha<\frac{2\phi_{3}}{\pi}$. If $u_0, u_1\in L^3(\Omega)$, then there exists a $L^{3}(\Omega)$-valued continuous function $u$ given by 
$$u(t,x)=E_{\alpha}(t\mathcal{A}_3)u_0(x) + S_{\alpha}(t\mathcal{A}_3)u_1(x) + \int_0^t R_{\alpha}((t-s)\mathcal{A}_3)u(s,x)^{3}\,ds,\quad t\in [0,\tau]\ \mbox{and}\ x\in\Omega,$$
which is the unique $\eps$-regular mild solution to \eqref{diffwaveexample} satisfying
$$\lim_{t\to0^+} t^{\alpha\varepsilon}\|u(t)\|_{X^{1+\varepsilon}_{3}}=0,$$
for any $\eps \in (0,\frac{1}{3\alpha})$. Furthermore, if $0\le \theta < 3\varepsilon$, then 
$u \in C\left((0,\tau_0];X^{1+\theta}_{3}\right).$  Finally, $u$ can be continued to a maximal time of existence $\tau_{max}>0$ such that
$$\limsup_{t\rightarrow \tau_{max}^{-}}\int_{\Omega}|u(t,x)|^{3}dx=+\infty,$$
if $\tau_{max}<+\infty$.

	In Theorem \ref{local} we consider the initial data $u_0$ and $u_1$ belonging to $X^{1}_{q}=L^{q}(\Omega)$. However, the same conclusions are true if $u_0\in L^{q}(\Omega)$, and the second initial condition has even less regularity, see Remark \ref{rmk}.

Also, we are interested on studying the issue of global solutions to \eqref{diffwave}, and  the asymptotic behavior of these solutions. We obtain the following global small data result.

\begin{theorem}\label{globalexistence}
	Let $1<\alpha<\frac{2\phi_q}{\pi}$, $1<\rho, q<\infty$, such that $q=\frac{N(\rho-1)}{2}$, and  $0<\varepsilon<\frac{N}{N+2q}$,  with $\alpha\rho\varepsilon<1$.  If $u_0\in L^{q}(\Omega)$ and $u_1\in X^{1-\frac{1}{\alpha}}_{q}$ have sufficiently small norms, then there exists $\mu>0$ and a globally defined $\eps$-regular mild solution $u=u(\cdot\ ;u_0,u_1)$ to the problem \eqref{diffwavep} that verifies
	$$t^{\alpha\eps}\|u(t)\|_{X^{1+\eps}_{q}}\le \mu,\quad t>0.$$
	Moreover, if $u_0, w_0\in L^{q}(\Omega)$ and $ u_1, w_1\in X^{1-\frac{1}{\alpha}}_{q}$ have sufficiently small norms, then there exists a constant $c>0$ such that
	$$t^{\alpha\eps}\|u(t,u_0,u_1)-u(t,w_0,w_1)\|_{X^{1+\eps}_{q}}\le c\left(\|u_0 - w_0\|_{L^{q}(\Omega)}+\|u_1 - w_1\|_{X^{1-\frac{1}{\alpha}}_{q}}\right),\quad t>0.$$
\end{theorem}

Our next result addresses the subject of stability for global solutions. As far as we know, this work is the first
to deal with problem \eqref{diffwave} in this context.

\begin{theorem}\label{asymptotics}
	Under the conditions of Theorem \ref{globalexistence}, let $u=u(\cdot\ ;u_0,u_1)$ and $v=v(\cdot\ ;v_0,v_1)$ be $\eps$-regular mild solutions to problem \eqref{diffwavep}. Then 
	\begin{equation}\label{est1}
		\lim_{t\to+\infty}t^{\alpha\eps}\left\| E_\alpha(t\mathcal{A}_q)(u_0 - v_0) + S_\alpha(t\mathcal{A}_q)(u_1 - v_1)  \right\|_{X^{1+\eps}_{q}}=0
	\end{equation}
	if, and only if, 
	\begin{equation}\label{est2}
		\lim_{t\to+\infty}t^{\alpha\eps}\left\| u(t)-v(t)  \right\|_{X^{1+\eps}_{q}}=0.
	\end{equation}
\end{theorem}

As application of the above results, we deal with the asymptotic analysis of a concrete situation. Consider the Cauchy problem
\begin{equation}\label{diffwave3/2}
	\left\{ \begin{array}{ll}
		\partial_{t}^{\frac{3}{2}} u = \Delta u +u\sqrt[3]{u^4},~\mbox{in}~[0,\infty)\times \Omega,\\
		u=0,~\mbox{on}~[0,\infty)\times \partial\Omega,~~\\
		u(0,x) = u_0(x),~ u'(0,x) = u_1(x),~\mbox{in}~\Omega,&
	\end{array} \right.
\end{equation}
where $\Omega$ is bounded open subset of $\mathbb{R}^{3}$ with sufficiently smooth boundary $\partial \Omega$, $\partial_{t}^{\frac{3}{2}}$ is  Caputo's fractional derivative of order $\frac{3}{2}$. Our above results ensure that if $u_{0}\in L^2(\Omega)$ and $u_1\in X^{\frac{1}{3}}_{2}$ have sufficiently small norm, then for every {$\eps\in(0,\frac{2}{7})$}, problem \eqref{diffwave3/2} has a globally defined $\eps$-regular mild solution given by
$$u(t,x)=E_{\frac{3}{2}}(t\mathcal{A}_2)u_0(x) + S_{\frac{3}{2}}(t\mathcal{A}_2)u_1(x) + \int_0^t R_{\frac{3}{2}}((t-s)\mathcal{A}_2)u(s,x)|u(s,x)|^{\frac{4}{3}}\,ds,\quad t>0\ \mbox{and}\ x\in\Omega.$$
Furthermore, given $\phi,\psi\in C^{\infty}_{0}$, let $v$ be a globally defined $\eps$-regular mild solution of \eqref{diffwave3/2} with initial condition $v_0=u_0+\phi$ and $v_1=u_1+\psi$.  Then, for every {$\eps\in(0,\frac{2}{7})$}, it follows that 
$$	\lim_{t\to+\infty}t^{\frac{3\eps}{2}}\left\| u(t)-v(t)  \right\|_{X^{1+\eps}_{q}}=0.$$
In fact, taking $\beta<\frac{1}{3}$, it is sufficient to note that
\begin{eqnarray*}
	\left\| E_{\frac{3}{2}}(t\mathcal{A}_2)(u_0-v_0) + S_{\frac{3}{2}}(t\mathcal{A}_2)(u_1-v_1)  \right\|_{X^{1+\eps}_{2}}& = &\left\| E_{\frac{3}{2}}(t\mathcal{A}_2)\phi + S_{\frac{3}{2}}(t\mathcal{A}_2)\psi  \right\|_{X^{1+\eps}_{2}}\\
	& \le & \left\| E_{\frac{3}{2}}(t\mathcal{A}_2)\phi\right\|_{X^{1+\eps}_{2}} + \left\|S_{\frac{3}{2}}(t\mathcal{A}_2)\psi  \right\|_{X^{1+\eps}_{2}}\\
	& \le & Mt^{1-\frac{3}{2}(1+\eps-\beta)} \left\| \phi\right\|_{X^{\beta}_{2}} + Mt^{-\frac{3}{2}(1+\eps-\beta)}\left\|\psi  \right\|_{X^{\beta}_{2}}\\
\end{eqnarray*}
since $-\frac{3}{2}(1-\beta)<1-\frac{3}{2}(1-\beta)<0$,  we have
$$	\lim_{t\to+\infty}t^{\frac{3\eps}{2}}\left\| E_{\frac{3}{2}}(t\mathcal{A}_2)(u_0-v_0) + S_{\frac{3}{2}}(t\mathcal{A}_2)(u_1-v_1)  \right\|_{X^{1+\eps}_{2}}=0,$$
and the result follows from Theorem \ref{asymptotics}. Consequently, using the above results we present a basin of attraction around each $\eps$-regular mild solution of \eqref{diffwave3/2}.

The organization of this paper is as follows.  In the next section, we rewrite the semilinear fractional diffusion-wave equation \eqref{diffwave} as an abstract evolution problem in the scale $\{X^{\gamma}_{q}\}_{\gamma\in\mathbb{R}}$, remembering some important known results. In Section \ref{proofmain}, we present the proofs of our main results.

\section{Abstract setting of the problem and known results } \label{setting}

In order to fit the problem in an abstract functional setting let us recall the usual functional spaces involved.  Remember that
the operator $L = - \Delta$ with Dirichlet boundary conditions can be seen as a sectorial operator in $E^{0}_{q} = L^q(\Omega)$ with domain $E^{1}_{q} = W^{2,q}(\Omega) \cap W^{1,q}_{0}(\Omega)$.  It is well known that the scale of fractional powers spaces $\{E^{\gamma}_{q} \}_{\gamma \in \mathbb{R}}$ associated with $L$
checks
\begin{equation*}
	\left\{
	\begin{matrix}
		E^{\gamma}_{q} \hookrightarrow H^{2\gamma}_{q}(\Omega),&\gamma\ge0,&1<q<\infty,~~~~~~~~~ ~~~~~\\
		E^{-\gamma}_{q} = (E^{\gamma}_{q'})',&\gamma\ge0,&1<q<\infty, ~~ q'=\frac{q}{q-1},
	\end{matrix}
	\right.
\end{equation*}
see \cite{Amann}. Therefore, using duality arguments, we obtain
\begin{equation}\label{potfrac}
	\left\{
	\begin{matrix}
		E^{\gamma}_{q} \hookrightarrow L^r(\Omega),&r\le\frac{Nq}{N-2\gamma q},&0\le \gamma<\frac{N}{2q},~~~~~\\
		E^{0}_{q} = L^q(\Omega),~& \\
		E^{\gamma}_{q} \hookleftarrow L^s(\Omega),&s\ge\frac{Nq}{N-2\gamma q},&-\frac{N}{2q'}<\gamma\le 0.~~~  
	\end{matrix}
	\right.
\end{equation}
Moreover, the realization of $L$ in $E^{\gamma}_{q}$, denoted by $L_{\gamma}$, is an isometry from $E^{\gamma+1}_{q}$ into $E^{\gamma}_{q}$,  
\begin{equation*}
	L_\gamma : D(L_\gamma) = E^{\gamma+1}_{q} \subset E^{\gamma}_{q}  \to E^{\gamma}_{q}, 
\end{equation*}
a sector operator, and $D(L^{\alpha}_{\gamma}) = E^{\gamma+\alpha}_{q}$.

Our goal is to consider \eqref{diffwave} with initial data belonging to $L^q(\Omega)$, with $1 < q <\infty$ and $q = \frac{N(\rho - 1)}{2}$. Hence,  set $X^{\gamma}_{q}:= E^{\gamma-1}_{q}$, $\gamma \in \mathbb{R}$, and let $\mathcal{A}_q : X^{1}_{q} \subset X^{0}_{q} \to X^{0}_{q}$ be the operator $L_{-1}$. From \eqref{potfrac},
it follows that the scale of fractional powers spaces associated with $\mathcal{A}_q$ satisfies
\begin{eqnarray}\label{potfracmodLQ}
	\left\{\begin{array}{lrc}
		X_{q}^{\beta}\hookrightarrow L^r(\Omega), \ \ r\leq \frac{Nq}{N+2q-2\beta q}, \ \ 1\leq\beta<1+\frac{N}{2q},\\
		X_{q}^{1}=L^q(\Omega),\\
		X_{q}^{\beta}\hookleftarrow L^s(\Omega), \ \ s\geq \frac{Nq}{N+2q-2\beta q}, \ \ 1-\frac{N}{2q'}<\beta\leq 1.\\
	\end{array}\right.
\end{eqnarray}
 Using this framework, problem (\ref{diffwave}) can be rewritten in $X^{0}_{q}$ as
\begin{equation}\label{diffwavep}
	\left\{ \begin{array}{ll}
		D_{t}^{\alpha} u = \mathcal{A}_q u +f(u),\\
		u(0) = u_0,~ u'(0) = u_1,&
	\end{array} \right.
\end{equation}
where $u_0,u_1\in X^{1}_{q}=L^{q}(\Omega)$. 

Before proceeding, we need to be more precise about the meaning of solution to  \eqref{diffwavep}, or equivalently to \eqref{diffwave}, we are looking for. 
Let $(E_\alpha(t\mathcal{A}_q))_{t\ge 0}$, $(S_\alpha(t\mathcal{A}_q))_{t\ge 0}$ and $(R_\alpha(t\mathcal{A}_q))_{t\ge 0}$  be the  Mittag-Leffler operators generated by $\mathcal{A}_q$. 

\begin{de}
A continuous function $u:[0,\tau] \to L^q(\Omega)$ is an $\varepsilon$-regular mild solution to \eqref{diffwavep} if $u \in {C}((0,\tau]; X^{1+\varepsilon}_{q})$ and
	\begin{equation*}
		u(t)=E_\alpha(t\mathcal{A}_q)u_0 + S_\alpha(t\mathcal{A}_q)u_1 + \int_0^t R_{\alpha}((t-s)\mathcal{A}_q)f(u(s))\,ds,
	\end{equation*}
for all $t\in [0,\tau]$.
\end{de}


Remembering that $\mathcal{A}_q : X^{1}_{q} \subset X^{0}_{q} \to X^{0}_{q}$ is a sectorial operator, we can set $\phi_q\in (\frac{\pi}{2}, 
\pi)$  such that $\mathbb{S}_{\phi_q}=\{z\in\mathbb{C}: |\arg(z)|\leq\phi_q, z\neq 0\}\subset\rho(\mathcal{A}_q)$, see \cite[Section 6.5]{CLR}. 
From now on, we consider $\alpha\in (1,\frac{2\phi_q}{\pi})$. Note that if $\eta_0\in (\frac{\pi}{2}, \frac{\phi_q}{\alpha})$ and $\lambda\in\mathbb{S}_{\eta_0}$ then 
$$\lambda^{\alpha}\in\mathbb{S}_{\phi_q}\subset\rho(A_q)\quad \mbox{and}\quad \|(\lambda^\alpha-\mathcal{A}_q)^{-1}\|_{\mathcal{L}(X_{q}^{1})}\leq C|\lambda|^{-\alpha}.$$
In \cite{deAS},  the authors prove that the families of bounded linear operators $(E_\alpha(t\mathcal{A}_q))_{t\ge 0}$, $(S_\alpha(t\mathcal{A}_q))_{t\ge 0}$ and $(R_\alpha(t\mathcal{A}_q))_{t\ge 0}$ are strongly continuous and admit analytic extensions to suitable sectors of the complex plane. The following smoothing effect property  is central to ensure existence and uniqueness of $\varepsilon$-regular mild solution to \eqref{diffwavep}.

\begin{lema}[\cite {deAS}]
	Consider $1<\alpha<\frac{2\phi_q}{\pi}$ and $0\le\theta<\beta\le 1$. There exists $M>0$ such that
	\begin{equation*}
		\|E_{\alpha}(t\mathcal{A}_q)x\|_{X^{1+\theta}_{q}}\leq Mt^{-\alpha(1+\theta-\beta)}\|x\|_{X^{\beta}_{q}},\quad \|S_{\alpha}(t\mathcal{A}_q)x\|_{X^{1+\theta}_{q}}\leq Mt^{1-\alpha(1+\theta-\beta)}\|x\|_{X^{\beta}_{q}}
	\end{equation*}
	and
	\begin{equation*}
		\|R_{\alpha}(t\mathcal{A}_q)x\|_{X^{1+\theta}_{q}}\leq Mt^{-1-\alpha(\theta-\beta)}\|x\|_{X^{\beta}_{q}},
	\end{equation*}
	for all $x\in X^{\beta}_{q}$ and $t>0$. 
\end{lema}

It follows from the above estimate that for all $t>0$, the operators $t^{\alpha\theta}E_{\alpha}(t\mathcal{A}_q):X^{1}_{q}\to X^{1+\theta}_{q}$ are bounded linear operators satisfying
$$\|t^{\alpha\theta}E_{\alpha}(t\mathcal{A}_q)\|_{\mathcal{L}(X^{1}_{q},X^{1+\theta}_{q})}\le M,$$
with $M>0$ independent of $t>0$, for all $0\le\theta\le 1$. Moreover, given a compact subset $J$ of $X^{1}_{q}$, we have
$$\lim_{t\to0^{+}}\sup_{x\in J}\|t^{\alpha\theta}E_{\alpha}(t\mathcal{A}_q)x\|_{X^{1+\theta}_{q}}=0.$$
Similar properties can be proved to $(S_{\alpha}(t\mathcal{A}_q))_{t\ge0}$ and $(R_{\alpha}(t\mathcal{A}_q))_{t\ge0}$, see \cite{deAS}.

\begin{remark}\label{rmkcritico}
We now revisit the discussion on the critical value \eqref{critico}. On the one hand, the strategy adopted in \cite{deAS} was, considering $\beta > 0$ such that $1-\frac{N}{2q'}<\beta<1$ and $1<\rho\leq 1+\frac{2q}{N}(1-\beta )$, to prove that the function $f$ in \eqref{diffwavep} is well-defined from $X^{1}_{q}$ into $X^{\beta}_{q}$, and is a locally Lipschitz function. Then, they define, on a suitable complete metric space, the nonlinear map $T$  by
$$Tu(t)=E_\alpha(t\mathcal{A}_q)u_0 + S_\alpha(t\mathcal{A}_q)u_1 + \int_0^t R_{\alpha}((t-s)\mathcal{A}_q)f(u(s))\,ds,\quad t\in [0,\tau].$$
Using the linear estimates presented in the above lemma and the locally Lipschitz property, the authors were led to analyze the convergence of the integral 
$$\int_{0}^{t}(t-s)^{\alpha\beta-1}ds,$$
which is equivalent to the fact that $\beta>0$\footnote{Note that $\beta>0$ implies $1<\rho\leq 1+\frac{2q}{N}(1-\beta )<1+\frac{2q}{N}.$}. That is, the entire argument collapses when $\beta=0$. 

On the other hand, it is important to highlight that the situation described above is not merely a technical issue. Indeed, the locally Lipschitz property is not a sufficient hypothesis to guarantee the existence of mild solutions for \eqref{diffwavep}. For instance, if we consider the function $f(u)=-2\mathcal{A}_qu$, which is globally Lipschitz from $X^{1}_{q}$ into $X^{0}_{q}$, the correspondent problem is given by
$$D^{\alpha}_{t}u=-\mathcal{A}_qu,$$ 
which is not locally well-posed. This argumentation underscores the critical role played by \eqref{critico}.
\end{remark}

We close this section studying some properties of the nonlinear term $f$ on  the scale of fractional powers spaces $\{X^{\gamma}_{q} \}_{\gamma \in \mathbb{R}}$ associated with ${\cal A}_q$. 
\begin{lema}\label{reg} 
Let $\rho>1$, $1 < q <\infty$ and $q = \frac{N(\rho - 1)}{2}$. For all $0< \varepsilon < \frac{N}{N+2q}$ the function
$$f:X^{1+\varepsilon}_{q}\to X^{\rho\varepsilon}_{q}$$
is well defined, and verifies
$$\|f(u)-f(v)\|_{X^{\rho\varepsilon}_{q}} \le c\left(\| u\|^{\rho-1}_{X^{1+\varepsilon}_{q}}+\| v\|^{\rho-1}_{X^{1+\varepsilon}_{q}}\right)\|u-v\|_{X^{1+\varepsilon}_{q}}$$
and 
$$\|f(u)\|_{X^{\rho\varepsilon}_{q}} \le c\|u\|_{X^{1+\varepsilon}_{q}}^{\rho}$$
for some $c>0$.
\end{lema}

\begin{proof}
From \eqref{potfracmodLQ}, it follows that
$$	X_{q}^{1+\epsilon}\hookrightarrow L^{\frac{Nq}{N-2\epsilon q}}(\Omega)\quad \mbox{and}\quad L^{\frac{Nq}{N+2q-2\rho\epsilon q}}(\Omega) \hookrightarrow X_{q}^{\rho\epsilon}.$$
Since
$$	N+2q-2\rho\epsilon q	 = N  + N(\rho-1) - 2\rho\epsilon q	=(N-2\epsilon q)\rho $$
we have $$\frac{Nq}{N+2q-2\rho\epsilon q} = \frac{Nq}{\rho(N-2\epsilon q)}.$$
Then, using \eqref{fc} we get
\begin{align*}
 \|f(u)-f(v)\|_{X^{\rho\epsilon}}  & \leq   \|f(u)-f(v)\|_{L^{\frac{Nq}{\rho(N-2\epsilon q)}}(\Omega)}\\
 & \le c\|u-v\|_{L^{\frac{Nq}{N-2\epsilon q}}(\Omega)}\left(\|u\|_{L^{\frac{Nq}{N-2\epsilon q}}(\Omega)}^{\rho-1}+\|v\|_{L^{\frac{Nq}{N-2\epsilon q}}(\Omega)}^{\rho-1}\right)\\
 & \le  \|u-v\|_{X_{q}^{1+\epsilon}}\left(\|u\|_{X_{q}^{1+\epsilon}}^{\rho-1}+\|v\|_{X_{q}^{1+\epsilon}}^{\rho-1}\right).
\end{align*}
Taking $v=0$,  we conclude the proof.
\end{proof}

\section{Proof of main results}\label{proofmain}
\subsection{Local well-posedness}

\begin{proof}[Proof of Theorem \ref{local}] 	We divide this proof in 3 main parts.

\noindent \textbf{Part 1 - Existence of $\varepsilon$-regular mild solution, proofs of (A) and (B):} Define $\mu>0$  by 
	$$Mc\mu^{\rho-1}B(\alpha(\rho\varepsilon - \varepsilon),1-\alpha\rho\varepsilon)=\frac{1}{4},$$
and choose $r=r(\mu,M)$ such that 
	$$r=\frac{\mu}{4M}.$$
For $v_0\in L^q(\Omega)$ fixed, choose $\tau_0\in(0,1]$ such that 
$$	t^{\alpha \varepsilon} \|E_\alpha(t\mathcal{A}_q)v_0\|_{X^{1+\varepsilon}_{q}} \le \frac{\mu}{4}\quad \mbox{and}\quad t \|u_1\|_{L^{q}(\Omega)} \le \frac{\mu}{4},$$
for all $0\le t\le \tau_0$. Consider
	\begin{equation*}
		K(\tau_0)= \left\{ u \in \mathcal{C}((0,\tau_0];X^{1+\varepsilon}_{q}) : \sup_{t \in (0,\tau_0]} t^{\alpha \varepsilon} \|u(t)\|_{X^{1+\varepsilon}_{q}} \le \mu \right\},
\end{equation*}
which the metric
\begin{equation*}
\mathrm{d}(u,v) = \sup_{t \in (0,\tau_0]} t^{\alpha \varepsilon} \|u(t)-v(t)\|_{X^{1+\varepsilon}_{q}},
\end{equation*}
and define the map $T$ on $K(\tau_0)$ by
\begin{equation*}
	Tu(t)= E_\alpha(t\mathcal{A}_q)u_0 + S_\alpha(t\mathcal{A}_q)u_1 + \int_0^t R_{\alpha}((t-s)\mathcal{A}_q)f(u(s))\,ds, ~t \in (0,\tau_0].
\end{equation*}

Let us prove that $T:K(\tau_0) \to K(\tau_0)$ is a contraction. 
Consider $t_1, t_2 \in (0,\tau_0]$ such that $t_1>t_2$. For any $0\le \theta < \rho\varepsilon$ and $u \in K(\tau_0)$, we have
			\begin{eqnarray*}
			 \|(Tu)(t_1)-(Tu)(t_2)\|_{X^{1+\theta}_{q}}&\le& \|E_\alpha(t_1\mathcal{A}_q)u_0 - E_\alpha(t_2\mathcal{A}_q)u_0\|_{X^{1+\theta}_{q}}   +  \|S_\alpha(t_1\mathcal{A}_q)u_1 - S_\alpha(t_2\mathcal{A}_q)u_1\|_{X^{1+\theta}_{q}} \\
			&  + & \mathcal{I}, 
			\end{eqnarray*}
			where
			\begin{equation*}
				\mathcal{I}=\left\|  \int_{0}^{t_{1}} R_{\alpha}((t_1-s)\mathcal{A}_q)f(u(s))\,ds - \int_{0}^{t_{2}} R_{\alpha}((t_2-s)\mathcal{A}_q)f(u(s))\,ds\right\|_{X^{1+\theta}_{q}}.  
			\end{equation*}
			Since $(E_{\alpha}(t\mathcal{A}_q))_{t\ge0}$ and $(S_{\alpha}(t\mathcal{A}_q))_{t\ge0}$ are strongly continuous, we have that 
			$$\|E_\alpha(t_1\mathcal{A}_q)u_0 - E_\alpha(t_2\mathcal{A}_q)u_0\|_{X^{1+\theta}_{q}}\to 0  \quad \mbox{and} \quad   \|S_\alpha(t_1\mathcal{A}_q)u_1 - S_\alpha(t_2\mathcal{A}_q)u_1\|_{X^{1+\theta}_{q}}\to 0,$$
			as $t_{1}\to t_{2}^{+}$.		To show that $\mathcal{I}$ has this same property,  note that
			\begin{eqnarray*}
				\mathcal{I} &\le&  \int_{0}^{t_{2}} \| R_{\alpha}((t_2-s)\mathcal{A}_q)f(u(s)) - R_{\alpha}((t_1-s)\mathcal{A}_q)f(u(s))\|_{X^{1+\theta}_{q}}\,ds\\
				& + &  \int_{t_2}^{t_{1}} \| R_{\alpha}((t_1-s)\mathcal{A}_{q})f(u(s))\|_{X^{1+\theta}_{q}}\,ds.
			\end{eqnarray*}
The strong continuity of $(R_{\alpha}(t\mathcal{A}_q))_{t\ge0}$ and the Lebesgue's dominated convergence theorem ensure that the first term on the right side goes to $0$ as $t_{1}\to t_{2}^{+}$.	For the second term, we have
			\begin{eqnarray*}
 \int_{t_2}^{t_{1}} \|R_{\alpha}((t_1-s)\mathcal{A}_q)f(u(s))\|_{X^{1+\theta}_{q}}\,ds
				\le  M\int_{t_2}^{t_{1}} (t_1 -s)^{-1-\alpha(\theta - \rho\varepsilon)}\|f(u(s))\|_{X^{\rho\varepsilon}_{q}}\,ds
			\end{eqnarray*}
			\begin{eqnarray*}
				&\le & Mc\int_{t_2}^{t_{1}} (t_1 -s)^{-1-\alpha(\theta - \rho\varepsilon)}\|u(s)\|_{X^{1+\varepsilon}_{q}}^{\rho}\,ds\\
				&\le &  Mc \mu^{\rho}t_{1}^{-\alpha\theta}\int_{t_2/t_1}^{1} (1 -s)^{-1-\alpha(\theta - \rho\varepsilon)}s^{-\alpha\rho\varepsilon}\,ds\\
			\end{eqnarray*}
			and so, we have to
			\begin{equation*}
				\lim_{t_1 \to t_2^+}  \int_{t_2}^{t_{1}} \|R_{\alpha}((t_1-s)\mathcal{A}_q)f(u(s))\|_{X^{1+\theta}_{q}}\,ds = 0.
			\end{equation*}
The case $t_1< t_2 $ is similar. To conclude that $K(\tau_0)$ is $T$-invariant, let $t \in (0,\tau_0]$ and $u \in K(\tau_0)$. Then
			\begin{eqnarray*}
				t^{\alpha \varepsilon}\|Tu(t)\|_{X^{1+\varepsilon}_{q}} &\le&   t^{\alpha \varepsilon}\|E_\alpha(t\mathcal{A}_q)u_0 \|_{X^{1+\varepsilon}_{q}} +  t^{\alpha \varepsilon}\|S_\alpha(t\mathcal{A}_q)u_1\|_{X^{1+\varepsilon}_{q}}\\
				& + & t^{\alpha \varepsilon}\int_0^t \|R_{\alpha}((t-s)\mathcal{A}_q)f(u(s))\|_{X^{1+\varepsilon}_{q}}\,ds\\
				&\le& t^{\alpha \varepsilon}\|E_\alpha(t\mathcal{A}_q)u_0 - E_\alpha(t\mathcal{A})v_0\|_{X^{1+\varepsilon}_{q}} +  t^{\alpha \varepsilon}\|E_\alpha(t\mathcal{A}_q)v_0\|_{X^{1+\varepsilon}_{q}}\\
				 & + & t^{\alpha \varepsilon}\|S_\alpha(t\mathcal{A}_q)u_1\|_{X^{1+\varepsilon}_{q}}
				 + t^{\alpha \varepsilon}\int_0^t \|R_{\alpha}((t-s)\mathcal{A}_q)f(u(s))\|_{X^{1+\varepsilon}_{q}}\,ds\\
				&\le& \frac{3\mu}{4} + Mc  \int_0^t (t-s)^{-1 - \alpha(\varepsilon - \rho\varepsilon)}\|u(s)\|_{X^{1+\varepsilon}_{q}}^{\rho}\,ds\\
				&\le& \frac{3\mu}{4} + Mc\mu^{\rho}B(\alpha(\rho\varepsilon - \varepsilon),1-\alpha\rho\varepsilon)	\le \mu,
			\end{eqnarray*}
			and so $T:K(\tau_0) \to K(\tau_0)$ is well defined.

	Now, observe that if $u,v \in K(\tau_0)$ and $t\in[0,\tau_0]$, then
	\begin{eqnarray*}
		t^{\alpha \varepsilon}  \|Tu(t) - Tv(t)\|_{X^{1+\varepsilon}_{q}}
		&\le& t^{\alpha \varepsilon}\int_0^t \|R_{\alpha}((t-s)\mathcal{A}_q)[f(u(s))-f(v(s))]\|_{X^{1+\varepsilon}_{q}}\,ds\\
		&\le& Mt^{\alpha \varepsilon}\int_0^t (t-s)^{-1-\alpha(\varepsilon-\rho\varepsilon)}\|f(u(s))-f(v(s))\|_{X^{\rho\varepsilon}_{q}}\,ds\\
&\le& Mct^{\alpha \varepsilon}\int_0^t (t-s)^{-1-\alpha(\varepsilon-\rho\varepsilon)}(\|u(s)\|_{X^{1+\varepsilon}}^{\rho - 1}+\|v(s)\|_{X^{1+\varepsilon}_{q}}^{\rho - 1})\|u(s)-v(s)\|_{X^{1+\varepsilon}_{q}}\,ds\\
		&\le& \left(2 Mc \mu^{\rho - 1} t^{\alpha \varepsilon}\int_0^t (t-s)^{-1-\alpha(\varepsilon-\rho\varepsilon)}s^{- \alpha\rho\varepsilon}\,ds\right)\mathrm{d}(u,v)\\
		&\le& \left(2 Mc \mu^{\rho - 1} B(\alpha(\rho\varepsilon - \varepsilon),1-\alpha\rho\varepsilon)	\right)\mathrm{d}(u,v)\le \frac{1}{2}\ \mathrm{d}(u,v).
	\end{eqnarray*}
By the Banach fixed point theorem, $T$ has a unique fixed point $u \in K(\tau_0)$. 

We proved that $u \in \mathcal{C}((0,\tau_0];X^{1+\theta}_{q}) $, for all $0\le \theta < \rho\varepsilon$. Furthermore, 	
$$\lim_{t\to0}t^{\alpha\theta}\|u(t)\|_{X^{1+\theta}_{q}}= 0,\quad \forall\ 0<\theta<\rho\epsilon. $$
Indeed, it follows that
\begin{eqnarray*}
	t^{\alpha\theta}\|u(t)\|_{X^{1+\theta}_{q}} & \le & t^{\alpha\theta}\|E_{\alpha}(t\mathcal{A}_{q})u_0\|_{X^{1+\theta}_{q}}+ t^{\alpha\theta}\|S_{\alpha}(t\mathcal{A}_{q})u_1\|_{X^{1+\theta}_{q}}\\
	& + & t^{\alpha\theta}\int_{0}^{t}\|R_{\alpha}((t-s)\mathcal{A}_{q})f( u(s))\|_{X^{1+\theta}_{q}}ds\\
	& \le & t^{\alpha\theta}\|E_{\alpha}(t\mathcal{A}_{q})u_0\|_{X^{1+\theta}_{q}}+ t^{\alpha\theta}\|S_{\alpha}(t\mathcal{A}_{q})u_1\|_{X^{1+\theta}_{q}}\\
	& + &Mc \mu^{\rho-1}(\alpha(\rho\epsilon-\theta),1-\alpha\rho\epsilon)\sup_{0<s\le
		t}\{s^{\alpha\epsilon}\|u(s)\|_{X^{1+\varepsilon}_{q}}\}.
\end{eqnarray*}
Therefore, if $\theta=\epsilon$ we deduce
\begin{eqnarray*}\label{supcon}
	t^{\alpha\epsilon}\|u(t)\|_{X^{1+\varepsilon}_{q}} & \le & t^{\alpha\epsilon}\|E_{\alpha}(t\mathcal{A}_{q})u_0\|_{X^{1+\varepsilon}_{q}}+ t^{\alpha\epsilon}\|S_{\alpha}(t\mathcal{A}_{q})u_1\|_{X^{1+\varepsilon}_{q}}
	 + \frac{1}{4}\sup_{0<s\le
		t}\{s^{\alpha\epsilon}\|u(s)\|_{X^{1+\varepsilon}_{q}}\}.
\end{eqnarray*}
from which we obtain
$$ \sup_{0<s\le t}\{s^{\alpha\epsilon}\|u(s)\|_{X^{1+\varepsilon}_{q}}\}\le \frac{4}{3}\left(\sup_{0<s\le t}\{s^{\alpha\epsilon}\|E_{\alpha}(s\mathcal{A}_{q})u_0\|_{X^{1+\varepsilon}_{q}}+s^{\alpha\epsilon}\|S_{\alpha}(t\mathcal{A}_{q})u_1\|_{X^{1+\varepsilon}_{q}}\}\right)\to0,$$
as $t$ goes to $0$.  

The above estimate also ensures that $\lim_{t\to0^{+}}\|u(t)-u_0\|_{L^{q}(\Omega)}=0.$ In fact, we have
\begin{eqnarray*}
	\|u(t)-u_0\|_{L^{q}(\Omega)} &\le & \|E_{\alpha}(t\mathcal{A}_{q})u_0-u_0\|_{L^{q}(\Omega)}+ \|S_{\alpha}(t\mathcal{A}_{q})u_1\|_{L^{q}(\Omega)}\\
	& + & \int_{0}^{t}\|R_{\alpha}((t-s)\mathcal{A}_{q})f(u(s))\|_{L^{q}(\Omega)}ds\\
\end{eqnarray*}
$$\le \|E_{\alpha}(t\mathcal{A}_{q})u_0-u_0\|_{L^{q}(\Omega)}+  t\|u_1\|_{L^{q}(\Omega)} +  McB(\alpha\rho\epsilon,1-\alpha\rho\epsilon)\big(\sup_{0<s\le t}\{s^{\alpha\epsilon}\|u(s)\|_{X^{1+\varepsilon}_{q}}\}\big)^\rho.$$
Therefore,  we have that $u=u(\cdot\ ;u_0, u_1)$ is an $\epsilon$-regular mild solution to the problem \eqref{diffwavep}. It follows from \eqref{supcon} that if $J \subset B_{r} (v_0) $ is compact, then
$$\lim_{t\to0^+}t^{\alpha\varepsilon}\sup_{u_0, u_1 \in J}\|u(t;u_0, u_1)\|_{X^{1+\varepsilon}_{q}}=0,$$
and this proves $(A)$. 

To prove $(B)$,  consider $u_0,u_1,w_0,w_1 \in B_{r}(v_0) \subset L^{q}(\Omega)$. For $t\in (0,\tau_0]$ and $\theta \in [0, \rho\varepsilon)$, it follows that
\begin{eqnarray*}
	t^{\alpha \theta}\|u(t;u_0,u_1)-u(t;w_0,w_1)\|_{X^{1+\theta}_{q}} &\le&  M(\|u_0-w_0\|_{L^{q}(\Omega)}+\|u_1-w_1\|_{L^{q}(\Omega})\nonumber \\
	& + & \Gamma_{\theta}(t)\sup_{0< t \le \tau_0} t^{\alpha\theta}\|u(t;u_0,u_1)-u(t;w_0,w_1)\|_{X^{1+\theta}_{q}},
\end{eqnarray*}
where
$$	\Gamma_{\theta}(t) = Mc \mbox{B}(\alpha(\rho\varepsilon-\theta),1-\alpha\rho\varepsilon)\left((\sup_{0< t \le \tau_0} t^{\alpha\theta}\|u(t;u_0,u_1)\|_{X^{1+\theta}_{q}})^{\rho-1} + (\sup_{0< t \le \tau_0} t^{\alpha\theta}\|u(t;w_0,w_1)\|_{X^{1+\theta}_{q}})^{\rho-1}\right).$$
Then,
\begin{eqnarray*}
	\sup_{0< t \le \tau_0} t^{\alpha \theta}\|u(t;u_0,u_1)-u(t;w_0,w_1)\|_{X^{1+\theta}_{q}} &\le&  M(\|u_0-w_0\|_{L^{q}(\Omega)}+\|u_1-w_1\|_{L^{q}(\Omega)})\\
	&+& \frac{3}{4}\sup_{0< t \le \tau_0} t^{\alpha\theta}\|u(t;u_0,u_1)-u(t;w_0,w_1)\|_{X^{1+\theta}_{q}},
\end{eqnarray*}
and so
\begin{eqnarray*}
	\sup_{0< t \le \tau_0} t^{\alpha \theta}\|u(t;u_0,u_1)-u(t;w_0,w_1)\|_{X^{1+\theta}_{q}} &\le&  4M(\|u_0-w_0\|_{L^{q}(\Omega)}+\|u_1-w_1\|_{L^{q}(\Omega)}).
\end{eqnarray*}
Consequently,
\begin{equation*}
	t^{\alpha \theta}\|u(t;u_0,u_1)-u(t;w_0,w_1)\|_{X^{1+\theta}_{q}}\le  \Bar{c}(\|u_0-w_0\|_{L^{q}(\Omega)}+\|u_1-w_1\|_{L^{q}(\Omega)}) ,
\end{equation*}
where,
\begin{equation*}
	\Bar{c}=M(1 + 4 \sup \{ \Gamma_{\theta}(t) : 0< t \le \tau_0 \} ).
\end{equation*}

\noindent \textbf{Part 2 - Proof of (C):} Let $\tau_1>0$ and $v:[0,\tau_1]\to X^{1}_{q}$ be an $\epsilon$-regular mild solution on $[0,\tau_1]$ for the problem \eqref{diffwavep} satisfying
$$\lim_{t\to0^+} t^{\alpha\varepsilon}\|v(t)\|_{X^{1+\varepsilon}_{q}}=0.$$
Then, we can choose $\tau \in (0,\min\{\tau_1,\tau_0\}]$ small enough so that
$$\sup_{0<t\leq  \tau}t^{\alpha\epsilon}\|v(t)\|_{X^{1+\epsilon}_{q}}\leq \mu.$$
Therefore, the constraints of the functions $v$ and $u(\cdot\,;u_0,u_1)$ to the interval $(0, \tau]$ belong to the set
$$K(\tau):=\left\lbrace \phi \in C\left((0, \tau];X^{1+\epsilon}_{q}\right);\, \sup_{0<t\leq  \tau}t^{\alpha\epsilon}\|\phi(t)\|_{X^{1+\epsilon}_{q}}\leq \mu\right\rbrace$$
and they are fixed points of the map $\widetilde{T}$ defined on $K(\tau)$ by
$$(\widetilde{T}\phi)(t)=E_\alpha(t\mathcal{A}_q)u_0 + S_\alpha(t\mathcal{A}_q)u_1 + \int_0^t R_{\alpha}((t-s)\mathcal{A}_q)f(\phi(s))\,ds, \quad t \in (0, \tau].$$
We note that $\tau>0$ was taken arbitrarily small, so the proof that $\widetilde{T}$ is a contraction on $K(\tau)$ is entirely analogous to the proof for $T$ on $K(\tau_0)$. Then, by the uniqueness of fixed points for $\widetilde{T}$ on $K(\tau)$, we have that $v(t)=u(t;u_0,u_1)$ for all $t \in (0, \tau]$. Actually, since $v(0)=u_0=u(0;u_0,u_1)$, the equality holds on $[0, \tau]$. On the other hand, denoting
$$\kappa:=\sup_{s \in \left[ \tau, \min\{ \tau_1,\tau_0\}\right]}\left(\|v(s)\|_{X^{1+\epsilon}_{q}}^{\rho-1}+\|u(s;u_0,u_1)\|_{X^{1+\epsilon}_{q}}^{\rho-1}\right)<\infty,$$
for each $t \in [ \tau,\min\{ \tau_1,\tau_0\}]$, we have
\begin{align*}
	&\|v(t)-u(t;u_0,u_1)\|_{X^{1+\epsilon}_{q}}\leq \int_{0}^{t}\left\|R_{\alpha}((t-s)\mathcal{A}_q)f\left(v(s)\right)-R_{\alpha}((t-s)\mathcal{A}_q)f\left(u(s;u_0,u_1)\right)\right\|_{X^{1+\epsilon}_{q}}ds\\
	\leq&\, M \int_{0}^{t}(t-s)^{-\alpha(1+\epsilon-\rho\varepsilon)}\|f(v(s))-f(u(s;u_0,u_1))\|_{X^{\rho\varepsilon}_{q}}ds\\
	\leq&\,
	Mc\int_{ \tau}^{t}(t-s)^{-\alpha(1+\epsilon-\rho\varepsilon)}\|v(s)-u(s;u_0,u_1)\|_{X^{1+\epsilon}_{q}}\left(\|v(s)\|_{X^{1+\epsilon}_{q}}^{\rho-1}+\|u(s;u_0,u_1)\|_{X^{1+\epsilon}_{q}}^{\rho-1}\right)ds\\
	\leq&\,
	Mc\kappa\int_{ \tau}^{t}(t-s)^{-\alpha(1+\epsilon-\rho\varepsilon)}\|v(s)-u(s;u_0,u_1)\|_{X^{1+\epsilon}_{q}}ds.
\end{align*}
It follows from Singular Grönwall's Inequality that
$$\|v(t)-u(t;u_0,u_1)\|_{X^{1+\epsilon}_{q}}=0, \quad \forall \,t \in[\tau,\min\{\tau_1,\tau_0\}].$$
Therefore, $v(t)=u(t;u_0,u_1)$ for all $t \in [0,\min\{ \tau_1,\tau_0\}]$, as (C) states.

\noindent \textbf{Part 3 - Proof of (D):} Let $u=u(\cdot\ ; u_0,u_1)$ be an $\varepsilon$-regular mild solution to problem \eqref{diffwavep} in $(0,\tau_0]$. Denote by
$\kappa = 1 + (\tau_{0} + 1)^{\alpha\varepsilon} \|u(\tau_0) \|_{X^{1+\varepsilon}_{q}}$
and choose $\tau_1 \in (\tau_0,\tau_0 + 1]$ arbitrarily close to $\tau_0$ such that
\begin{equation*}
	t^{\alpha \varepsilon}\|E_\alpha(t\mathcal{A}_q)u_0 - E_\alpha(\tau_0\mathcal{A}_q)u_0\|_{X^{1+\varepsilon}_{q}}\le \frac{1}{4},\quad 	t^{\alpha \varepsilon}\|S_\alpha(t\mathcal{A}_q)u_1 - S_\alpha(\tau_0\mathcal{A}_q)u_1\|_{X^{1+\varepsilon}_{q}} \le \frac{1}{4},
\end{equation*}
\begin{equation*}
	t^{\alpha \varepsilon}\left\|\int_0^{\tau_0}  [R_{\alpha}((t-s)\mathcal{A}_q)-R_{\alpha}((\tau_0-s)\mathcal{A}_q)]f(u(s))ds\right\|_{X^{1+\varepsilon}_{q}}  \le \frac{1}{4},  
\end{equation*}
$${Mc\kappa^{\rho} \int_{\tau_0 /t}^t (1-s)^{-1-\alpha(\varepsilon-\rho\varepsilon)}s^{-\alpha\rho\varepsilon}ds \le \frac{1}{4}}\quad \mbox{and}\quad Mc\kappa^{\rho-1}\bigg( \int_{\tau_0 / t}^{1} (1-s)^{-1-\alpha(\varepsilon-\rho\varepsilon)}s^{-\alpha\rho\varepsilon}ds\bigg)\le \frac{1}{4},$$
for all $t \in [\tau_0,\tau_1]$. Let $\mathcal{S}$ be the set of all $v\in \mathcal{C}((0,\tau_1]; X^{1+\varepsilon}_{q})$ such that $v(t)=u(t)$,  for  $t \in (0,\tau_0]$, and 
\begin{equation*}
 \sup_{t \in [\tau_0, \tau_1]} t^{\alpha \varepsilon}\|v(t) - u(\tau_0) \|_{X^{1+\varepsilon}_{q}} \le 1.    
\end{equation*}
Then $\mathcal{S}$ is a complete metric space with metric
$$
	d(v,w) = \sup_{t \in (0, \tau_1]} t^{\alpha \varepsilon}\|v(t)-w(t)\|_{X^{1+\varepsilon}_{q}}.
$$
Define in $\mathcal{S}$ the map 
\begin{equation*}
	Tv(t)= E_\alpha(t\mathcal{A}_q)u_0 + S_\alpha(t\mathcal{A}_q)u_1 + \int_0^t R_{\alpha}((t-s)\mathcal{A}_q)f(v(s))ds, \quad t \in (0,\tau_1].
\end{equation*}

Similarly to the Part 1,  we can prove that $Tv \in \mathcal{C}((0,\tau_1];X^{1+\theta}_{q})$, for all $\theta \in [0, \rho\varepsilon)$, and any $v \in \mathcal{S}$. Furthermore, if  $t\in (0,\tau_0]$, then $Tv(t) = Tu(t) = u(t)$. Thereby, to prove that $T v \in \mathcal{S}$, we just need to check that
\begin{equation*}
	\sup_{t \in [\tau_0,\tau_1]} \|Tv(t) - u(\tau_0)\|_{X^{1+\varepsilon}_{q}} \le 1.
\end{equation*}
Indeed, for $t \in [\tau_0,\tau_1]$, we obtain
\begin{eqnarray*}
t^{\alpha \varepsilon}\|Tv(t) - u(\tau_0)\|_{X^{1+\varepsilon}_{q}}	&\le&   t^{\alpha \varepsilon}\|E_\alpha(t\mathcal{A}_q)u_0 - E_\alpha(\tau_0\mathcal{A}_q)u_0\|_{X^{1+\varepsilon}_{q}} +  t^{\alpha \varepsilon}\|S_\alpha(t\mathcal{A}_q)u_1 - S_\alpha(\tau_0\mathcal{A}_q)u_1\|_{X^{1+\varepsilon}_{q}}\\
	& + & t^{\alpha \varepsilon} \left\|\int_0^t  R_{\alpha}((t-s)\mathcal{A}_q)f(v(s))ds - \int_0^{\tau_0}  R_{\alpha}((\tau_0-s)\mathcal{A}_q)f(u(s))ds\right\|_{X^{1+\varepsilon}_{q}}\\
	&\le& \frac{1}{4} +  \frac{1}{4}+t^{\alpha \varepsilon}\left\|\int_0^{\tau_0}  [R_{\alpha}((t-s)\mathcal{A}_q)-R_{\alpha}((\tau_0-s)\mathcal{A}_q)]f(u(s))ds\right\|_{X^{1+\varepsilon}_{q}}\\
	& + & t^{\alpha \varepsilon} \left\|\int_{\tau_0}^t  R_{\alpha}((t-s)\mathcal{A}_q)f(v(s))ds\right\|_{X^{1+\varepsilon}_{q}}\\
	&\le& \frac{3}{4}+ {Mc\kappa^{\rho} \int_{\tau_0 /t}^t (1-s)^{-1-\alpha(\varepsilon-\rho\varepsilon)}s^{-\alpha\rho\varepsilon}ds} \le 1,
\end{eqnarray*}
which proves that $ \mathcal{S}$ is $T$-invariant. Note also that for any $v, w \in \mathcal{S}$, and all $t \in [\tau_0, \tau_1]$, we have
\begin{eqnarray*}
t^{\alpha \varepsilon}  \|Tv(t) - Tw(t)\|_{X^{1+\varepsilon}_{q}}	&\le& t^{\alpha \varepsilon}\int_{\tau_0}^{\tau_1} \|R_{\alpha}((t-s)\mathcal{A}_q)[f(v(s))-f(w(s))]\|_{X^{1+\varepsilon}_{q}}ds\\
	&\le&2Mc{\kappa^{\rho-1}}\bigg( \int_{\tau_0 / t}^{1} (1-s)^{-1-\alpha(\varepsilon-\rho\varepsilon)}s^{-\alpha\rho\varepsilon}ds\bigg)\sup_{s \in (0,\tau_1]} s^{{\alpha\varepsilon}}\|v(s)-w(s)\|_{X^{1+\varepsilon}_{q}}\\
	&\le&\frac{1}{2}\sup_{s \in (0,\tau_1]}s^{{\alpha\varepsilon}}\|v(s)-w(s)\|_{X^{1+\varepsilon}_{q}}.
\end{eqnarray*}
Consequently, $T$ is a strict contraction on $\mathcal{S}$ and, by the Banach fixed point theorem, it has a unique fixed
point $v \in \mathcal{S}$,  which is a continuation of $u(\cdot\ ;u_0,u_1)$ on $[0, \tau_1]$.  

To prove the uniqueness of the continuation, suppose that $u=u(\cdot\ ;,u_0,u_1)$ admits another continuation $w$ defined on some interval $[0,\tau']$, with $\tau_1 \ge \tau'$. For each $t \in [\tau', \tau_1]$, we have
\begin{eqnarray*}
	 \|v(t) - w(t)\|_{X^{1+\varepsilon}_{q}} 
	&\le& M\int_{\tau_0}^{t} (t-s)^{-1-\alpha(\varepsilon-\rho\varepsilon)}\|f(v(s))-f(w(s))\|_{X^{\rho\varepsilon}_{q}}\,ds\\
	&\le& \bar{c}  \int_{\tau_0}^{t} (t-s)^{-1-\alpha(\varepsilon-\rho\varepsilon)}\|v(s)-w(s)\|_{X^{1+\varepsilon}_{q}}\,ds,
\end{eqnarray*}
where	$\bar{c} := Mc\sup_{s \in [\tau_0,\tau]} \{ \|v(s)\|_{X^{1+\varepsilon}_{q}}^{\rho - 1}+\|w(s)\|_{X^{1+\varepsilon}_{q}}^{\rho - 1}\}$. It follows from Singular Grönwall's Inequality that
\begin{equation*}
	\|v(t)-w(t)\|_{X^{1+\varepsilon}_{q}} = 0, ~\forall t \in [\tau_0 ,\tau'].
\end{equation*}
Since $v(t) = u(t)= w(t)$,  for all $ t \in [0,\tau_0]$, we conclude that $v(t)= w(t)$, for all $t \in [0, \tau']$.

To conclude the proof of (D),  let $u=u(\cdot\ ;u_0, u_1)$ be the $\varepsilon$-regular mild solution of problem \eqref{diffwavep} satisfying
$$\lim_{t\to0^+} t^{\alpha\varepsilon}\|u(t)\|_{X^{1+\varepsilon}_{q}}=0.$$
Consider $\tau_{max} > 0$ its  maximal time of existence. 	Suppose by contradiction that $\tau_{max}<+\infty$  and 
\begin{equation*}
	\sup_{t \in (0,\tau_{max})}t^{\alpha\varepsilon}\|u(t)\|_{X^{1+\varepsilon}_{q}} < +\infty.
\end{equation*}
Let $(t_n)_{n=1}^{\infty}$ a sequence in $(0, \tau_{max})$ with $\lim_{n \to \infty}t_n = \tau_{max}$. Given $m, n \in \mathbb{N}$, without loss of generality,  suppose $0<t_n < t_m < \tau_{max}$. Then,
\begin{eqnarray*}
	 \|u(t_m)-u(t_n)\|_{X^{1+\eps}_{q}} & \le & \|E_\alpha(t_m\mathcal{A}_q)u_0 - E_\alpha(t_n\mathcal{A}_q)u_0\|_{X^{1+\varepsilon}_{q}} + \|S_\alpha(t_m\mathcal{A}_q)u_1 - S_\alpha(t_n\mathcal{A}_q)u_1\|_{X^{1+\varepsilon}_{q}}\\
	&+& \left\|\int_{0}^{t_{m}} R_{\alpha}((t_m-s)\mathcal{A}_q)f(u(s))\,ds - \int_{0}^{t_{n}} R_{\alpha}((t_n-s)\mathcal{A}_q)f(u(s))\,ds\right\|_{X^{1+\varepsilon}_{q}}. 
\end{eqnarray*}
By strong continuity of $(E_{\alpha}(t\mathcal{A}_q))_{t\ge 0}$ and $(S_{\alpha}(t\mathcal{A}_q))_{t\ge 0}$,  the two first terms of
the right side in the above inequality go to zero as $m, n \to \infty$. In the same way, the strong continuity of $(R_{\alpha}(t\mathcal{A}_q))_{t\ge 0}$ combined with the Lebesgue's dominated convergence theorem ensure that the third therm  goes to $0$ as $m, n \to \infty$. Therefore, $(u(t_n))_{n=1}^{\infty} \subset X^{1+\varepsilon}_{q}$ is a Cauchy sequence  in a Banach space and thus there exists $\Bar{u} \in X^{1+\varepsilon}_{q}$ such that  $\lim_{n \to \infty} \| u(t_n) - \Bar{u}\|_{X^{1+\varepsilon}_{q}}=0$.
With this, we can extend $u$ to $[0, \tau_{max}]$ obtaining the equality
\begin{equation*}
	u(t)= E_\alpha(t\mathcal{A}_q)u_0 + S_\alpha(t\mathcal{A}_q)u_1 + \int_0^t R_{\alpha}((t-s)\mathcal{A}_q)f(u(s))\,ds, \quad \forall t \in [0,\tau_{max}],
\end{equation*}
which contradicts the maximality of $\tau_{max}$.
\end{proof}

\begin{remark}\label{rmk}
	In Theorem \ref{local} we consider the initial data $u_0$ and $u_1$ belonging to $X^{1}_{q}=L^{q}(\Omega)$. However, the same conclusions are true if $u_0\in L^{q}(\Omega)$, and the second initial condition has even less regularity. Indeed, if we consider $u_1\in X^{\beta}_{q}$, for any $\beta \in (1-\frac{1}{\alpha}, 1]$, then $$\alpha\eps+1-\alpha(1+\eps-\beta)=1-\alpha+\alpha\beta>0$$ and, consequently,
	we can choice $\tau_0\in(0,1]$ at the beginning of the proof of Theorem \ref{local} such that 
	$$ t^{1-\alpha+\alpha\beta} \|u_1\|_{X^{\beta}_{q}} \le \frac{\mu}{4},$$
	for all $0\le t\le \tau_0$. The sequence of the proof follows the same steps with some adaptations.
\end{remark}

\subsection{Global well-posedness and asymptotic behavior}

\begin{proof}[Proof of Theorem \ref{globalexistence}]
	This proof is very similar to Part 1 of the proof of Theorem \ref{local}. Hence, we only point out the differences.	Define $\mu>0$ by  
	$$ Mc\mu^{\rho-1}{\bf B}=\frac{1}{4},$$ 
	where ${\bf B}={ B}\left(\alpha\left(\rho\eps-\eps\right),1-\alpha\rho\eps\right)$.
	Consider the complete metric space
	$$K=\left\{u\in C((0,\infty), X^{1+\eps}_{q}): \sup_{t>0}t^{\alpha\eps}\|u(t)\|_{X^{1+\eps}_{q}}\leq\mu\right\}$$
	with metric $\mathrm{d}(u,v) = \sup_{t>0} t^{\alpha \varepsilon} \|u(t)-v(t)\|_{X^{1+\varepsilon}_{q}}$. Define the operator
	$$Tu(t)=E_\alpha(t\mathcal{A}_q)u_0 + S_\alpha(t\mathcal{A}_q)u_1 + \int_0^t R_{\alpha}((t-s)\mathcal{A}_q)f(u(s))\,ds,\quad t>0.$$
	The continuity of $Tu:(0,\infty)\to X^{1+\eps}_{q}$ is proved in the same way as in Theorem \ref{local}. The main difference is to ensure that $K$ is a $T$-invariant set and a contraction. Indeed, if $t>0$ and $M(\|u_0\|_{L^q(\Omega)}+\|u_1\|_{X^{1-\frac{1}{\alpha}}_{q}})\le\frac{\mu}{4}$, then
	\begin{eqnarray*}
		t^{\alpha\eps}	\|Tu(t)\|_{X^{1+\eps}_{q}} & \le & t^{\alpha\eps}\|E_\alpha(t\mathcal{A}_q)u_0\|_{X^{1+\eps}_{q}}+ t^{\alpha\eps} \|S_\alpha(t\mathcal{A}_q)u_1\|_{X^{1+\eps}_{q}}\\
		&+& t^{\alpha\eps} \int_{0}^{t}\|R_{\alpha}((t-s)\mathcal{A}_q)f(u(s))\|_{X^{1+\eps}_{q}}ds\\
		&\le& M\|u_0\|_{L^{q}(\Omega)}+M\|u_1\|_{X^{1-\frac{1}{\alpha}}_{q}}\\
		&+& Mct^{\alpha\eps} \int_{0}^{t}(t-s)^{-1-\alpha(\eps-\rho\eps)}s^{-\alpha\rho\eps}\left(s^{\alpha\eps}\|u(s)\|_{L^{q}(\Omega)}\right)^{\rho}ds\\
		&\le& M(\|u_0\|_{L^{q}(\Omega)}+\|u_1\|_{X^{1-\frac{1}{\alpha}}_{q}})+ Mc\mu^{\rho}{\bf B}\\
		&\le& \frac{\mu}{4}+\frac{\mu}{4}<\mu,
	\end{eqnarray*}
	which implies that $T:K\to K$ is well defined. Furthermore,
	\begin{eqnarray*}
		t^{\alpha\eps}	\|Tu(t)-Tv(t)\|_{X^{1+\eps}_{q}} & \le &  \left(2Mc\mu^{\rho-1}t^{\alpha\eps} \int_{0}^{t}(t-s)^{-1-\alpha(\eps-\rho\eps)}s^{-\alpha\rho\eps}ds\right) \mathrm{d}(u,v)\\
		&\le& \left(2Mc\mu^{\rho-1}{\bf B}\right)\mathrm{d}(u,v)\\
		&\le& \frac{1}{2}\ \mathrm{d}(u,v).
	\end{eqnarray*}
	This shows that $T$ is a $\frac{1}{2}$-contraction and the existence of a global mild solution follows from the Banach fixed point theorem. The continuous dependence is proved likewise the proof of Theorem \ref{local}.	
\end{proof}

\begin{proof}[Proof of Theorem \ref{asymptotics}]
Firstly, note that 
\begin{eqnarray}\label{est3}
t^{\alpha\eps}\left\| u(t)-v(t)  \right\|_{X^{1+\eps}_{q}} & \le & t^{\alpha\eps}\left\| E_\alpha(t\mathcal{A}_q)(u_0 - v_0) + S_\alpha(t\mathcal{A}_q)(u_1 - v_1)  \right\|_{X^{1+\eps}_{q}}\nonumber\\
& + & t^{\alpha\eps}\int_{0}^{t}\left\|R_{\alpha}((t-s)\mathcal{A}_q)\big(f(u(s))-f(v(s))\big)\right\|_{X^{1+\epsilon}_{q}}ds.
\end{eqnarray}
If \eqref{est1} is true, then the first term of the right  side in the above estimate goes to $0$ as $t\to+\infty$. To the last term, we have
\begin{align}\label{est4}
t^{\alpha\eps}\int_{0}^{t}&\left\|R_{\alpha}((t   -  s)\mathcal{A}_q)\big(f(u(s))-f(v(s))\big)\right\|_{X^{1+\epsilon}_{q}}ds  \le  M\int_{0}^{t}(t-s)^{-1-\alpha(\eps-\rho\eps)}\|f(u(s))-f(v(s))\|_{X^{\rho\eps}_{q}}ds\nonumber\\
  & \le  Mc\int_{0}^{t}(t-s)^{-1-\alpha(\eps-\rho\eps)}\big(\left\|u(s)\right\|^{\rho-1}_{X^{1+\epsilon}_{q}}+\left\|v(s)\right\|^{\rho-1}_{X^{1+\epsilon}_{q}}\big)\|u(s)-v(s)\|_{X^{1+\eps}_{q}}ds\nonumber\\
&  \le  2Mc\mu^{\rho-1}\int_{0}^{t}(t-s)^{-1-\alpha(\eps-\rho\eps)}s^{-\alpha\rho\eps}s^{\alpha\eps}\|u(s)-v(s)\|_{X^{1+\eps}_{q}}ds\nonumber\\
&  \le  2Mc\mu^{\rho-1}\int_{0}^{1}(1-s)^{-1-\alpha(\eps-\rho\eps)}s^{-\alpha\rho\eps}(st)^{\alpha\eps}\|u(st)-v(st)\|_{X^{1+\eps}_{q}}ds.
\end{align}
Let $(t_n)_{n=1}^{\infty}$ a sequence of real numbers such that $t_{n}\to+\infty$ as $n\to+\infty$. For each $n\in\mathbb{N}$, the function $f_{n}:(0,1)\to\mathbb{R}$ given by
$$f_{n}(s)=(1-s)^{-1-\alpha(\eps-\rho\eps)}s^{-\alpha\rho\eps}(st_{n})^{\alpha\eps}\|u(st_{n})-v(st_{n})\|_{X^{1+\eps}_{q}},$$
is a mensurable function with respect to the Borel measure. Furthermore, for all $n\in\mathbb{N}$, $f_n$ is a nonnegative function, and 
$$f_n(s)\le g(s)= 2\mu(1-s)^{-1-\alpha(\eps-\rho\eps)}s^{-\alpha\rho\eps},\quad \forall s\in(0,1).$$
Since $g:(0,1)\to\mathbb{R}$ is an integrable function we can conclude that
\begin{align*}
\limsup_{n\to\infty} \int_{0}^{1} f_{n}(s)ds &=  \limsup_{n\to\infty}  \int_{0}^{1}(1-s)^{-1-\alpha(\eps-\rho\eps)}s^{-\alpha\rho\eps}(st_{n})^{\alpha\eps}\|u(st_{n})-v(st_{n})\|_{X^{1+\eps}_{q}}ds\\
& \le \int_{0}^{1}(1-s)^{-1-\alpha(\eps-\rho\eps)}s^{-\alpha\rho\eps}\limsup_{n\to\infty} \big((st_{n})^{\alpha\eps}\|u(st_{n})-v(st_{n})\|_{X^{1+\eps}_{q}}\big)ds.
\end{align*}
Setting ${\bf L}= \limsup_{t\to+\infty}t^{\alpha\eps}\left\| u(t)-v(t)  \right\|_{X^{1+\eps}_{q}}$, it follows from \eqref{est3},  \eqref{est4}, and the above estimate that
$$0\le {\bf L}\le \big(2Mc\mu^{\rho-1}{\bf B} \big){\bf L} \Leftrightarrow 0\le \big(1-2Mc\mu^{\rho-1}{\bf B} \big){\bf L}\le 0.  $$
Remembering that $1-2Mc\mu^{\rho-1}{\bf B} >0$, we have
$${\bf L}= \limsup_{t\to+\infty}t^{\alpha\eps}\left\| u(t)-v(t)  \right\|_{X^{1+\eps}_{q}}=0.$$

Reciprocally, suppose ${\bf L}=0$. Then
\begin{align*}
 t^{\alpha\eps}\left\| E_\alpha(t\mathcal{A}_q)(u_0 - v_0) + S_\alpha(t\mathcal{A}_q)(u_1 - v_1)  \right\|_{X^{1+\eps}_{q}} & \le  t^{\alpha\eps}\left\| u(t)-v(t)  \right\|_{X^{1+\eps}_{q}}\\
& +  t^{\alpha\eps}\int_{0}^{t}\left\|R_{\alpha}((t-s)\mathcal{A}_q)\big(f(u(s))-f(v(s))\big)\right\|_{X^{1+\epsilon}_{q}}ds
\end{align*}
$$\le  t^{\alpha\eps}\left\| u(t)-v(t)  \right\|_{X^{1+\eps}_{q}}  +  2Mc\mu^{\rho-1}\int_{0}^{1}(1-s)^{-1-\alpha(\eps-\rho\eps)}s^{-\alpha\rho\eps}(st)^{\alpha\eps}\|u(st)-v(st)\|_{X^{1+\eps}_{q}}ds,$$
which implies 
$$\limsup_{t\to+\infty} t^{\alpha\eps}\left\| E_\alpha(t\mathcal{A}_q)(u_0 - v_0) + S_\alpha(t\mathcal{A}_q)(u_1 - v_1)  \right\|_{X^{1+\eps}_{q}}\le {\bf L}+ \big(2Mc\mu^{\rho-1}{\bf B} \big){\bf L}=0.$$
\end{proof}


\noindent{\bf Declarations:} 
\begin{itemize}
	\item	 Ethics approval and consent to participate: Not applicable.
	\item	Consent for publication: Not applicable.
	\item	Availability of data and materials: Not applicable.
	\item	Competing interests: The authors have no relevant financial or non-financial interests to disclose.
	\item	Funding: Bruno de Andrade is partially supported by CNPQ/Brazil (grant 310384/2022-2) and FAPITEC/Sergipe/Brazil (grant 019203.01303/2024-1).
	\item	Authors' contributions: M. Costa contributed to the implementation of the research, the analysis of the results, and the writing of the manuscript. C. Cuevas contributed to the implementation of the research, the analysis of the results, and the writing of the manuscript. B. de Andrade contributed to the implementation of the research, the analysis of the results, and the writing of the manuscript.
	\item	Acknowledgements:  Not applicable.
\end{itemize}

\bibliographystyle{amsplain}

\end{document}